\documentclass[letter,11pt]{amsart}
\usepackage[T1]{fontenc}
\usepackage{lmodern}
\usepackage[all]{xy}
\usepackage{amssymb, mathrsfs}
\usepackage{enumitem,nicefrac,mathtools}
\usepackage[british]{babel}
\usepackage{setspace}
\usepackage{xcolor}

\usepackage[pdftitle={Continuous Fields},
  pdfauthor={Rasmus Bentmann},
  pdfsubject={Mathematics}
]{hyperref}
\usepackage[lite]{amsrefs}
\usepackage{microtype}

\usepackage[foot]{amsaddr}

\usepackage{graphicx}
\usepackage{tikz}
\usetikzlibrary{matrix,arrows}

\usepackage{array}
\usepackage[labelformat=simple]{subfig}

\DeclareMathOperator{\Hom}{Hom}
\DeclareMathOperator{\Prim}{Prim}
\DeclareMathOperator{\Ext}{Ext}

\newcommand*{\KK}{\textup{KK}}
\newcommand*{\K}{\textup{K}}
\newcommand*{\E}{\textup{E}}

\newcommand*{\into}{\rightarrowtail}
\newcommand*{\prto}{\twoheadrightarrow}

\newcommand*{\Cont}{\mathrm{C}}    
   
\newcommand*{\Nattrafo}{\mathcal{NT}}
\newcommand*{\FK}{\textup{FK}}

\newcommand*{\Cstarsep}{\mathfrak{C^*sep}}
\newcommand*{\KKcat}{\mathfrak{KK}}
\newcommand*{\Ecat}{\mathfrak{E}}

\newcommand*{\Ab}{\mathfrak{Ab}}

\newcommand*{\GMod}[1]{\mathfrak{Mod}(#1)^{\mathbb{Z}/2}}

\newcommand*{\Z}{\mathbb{Z}}
\newcommand*{\N}{\mathbb{N}}

\newcommand*{\Loclo}{\mathbb{LC}}
\newcommand*{\Compacts}{\mathbb{K}}

\newcommand*{\Q}{\mathbb{Q}}

\newcommand*{\nb}{\nobreakdash}

\newcommand*{\Bootstrap}{{\mathcal B}}
\newcommand*{\Cuntz}{{\mathcal O}}

\newcommand*{\Cstar}{\texorpdfstring{$C^*$\nb-}{C*}}
\newcommand*{\Star}{\texorpdfstring{$^*$\nb-}{*-}}

\newcommand*{\blank}{\text{\textvisiblespace}}

\newcommand*{\defeq}{\mathrel{\vcentcolon=}}

\theoremstyle{plain}

\numberwithin{equation}{section}

\theoremstyle{plain}
\newtheorem{theorem}[equation]{Theorem}
\newtheorem{lemma}[equation]{Lemma}

\newtheorem{corollary}[equation]{Corollary}

\newtheorem{proposition}[equation]{Proposition}

\theoremstyle{definition}
\newtheorem{definition}[equation]{Definition}

\theoremstyle{remark}
\newtheorem{remark}[equation]{Remark}

\newtheorem{example}[equation]{Example}

\title[One-parameter continuous fields with rational $\K$-theory]{One-parameter continuous fields of Kirchberg algebras with rational $\K$-theory}

\author{Rasmus Bentmann}
\address{Department of Mathematical Sciences\\University of
Copenhagen\\Universitetsparken 5\\2100 Copenhagen \O \\Denmark}
\email{bentmann@math.ku.dk}

\author{Marius Dadarlat}
\address{Department of Mathematics\\
  Purdue University\\
  150 N.~University Street\\
  West Lafayette, IN 47907-2067\\
  USA}
\email{mdd@math.purdue.edu}

\subjclass[2010]{46L35, 46L80, 19K35, 46M20}

\keywords{Kirchberg algebras, Continuous fields, $\K$-theory}

\thanks{R.B. was supported by the Danish National Research Foundation through the Centre for Symmetry and Deformation (DNRF92) and by the Marie Curie Research Training Network EU-NCG. M.D. was partially supported by NSF grant \#DMS--1101305}

\begin{document}

\begin{abstract}
We show that 
 separable continuous fields over the unit interval whose fibers are stable Kirchberg algebras that satisfy the universal coefficient theorem in K$\K$-theory (UCT)
and have rational $\K$-theory groups are classified up to  isomorphism by filtrated $\K$-theory.
\end{abstract}

\maketitle

\section{Introduction}
The purpose of this paper is to investigate the classification problem for continuous fields of Kirchberg algebras over the unit interval by $\K$-theory invariants.
 It is natural to associate to a $C[0,1]$-algebra $A$ the family of all
exact triangles of $\Z/2$-graded $\K$-theory groups
\[
\xymatrix  @C=-1pc{
\K_*\bigl(A(U)\bigl)\ar[rr]& & \K_*\bigl(A(Y)\bigl)\ar[dl]\\
{}& \K_*\bigl(A(Y\setminus U)\bigl)\ar[ul]}
\]
where $Y$ is a subinterval of $[0,1]$ and $U$ is a relatively open subinterval of $Y$.
The family of these exact triangles are assembled into an invariant $\FK(A)$ called the filtrated $\K$-theory of $A$, see Definition~\ref{def:FK}.

In this article we exhibit several  classes of separable continuous fields over the unit interval whose fibers are stable UCT Kirchberg algebras and for which filtrated $\K$-theory is a complete invariant. In particular, we show that this is the case for fields which are stable under tensoring with the universal UHF-algebra.
A \Cstar{}algebra $D$ has rational $\K$-theory if $\K_*(D)\cong\K_*(D) \otimes \Q$.

\begin{theorem}\label{thm:main-result}
Let $A$ and $B$ be separable continuous fields over the unit interval whose fibers are stable Kirchberg algebras that satisfy the UCT
and have rational $\K$-theory groups. Then any isomorphism of filtrated $\K$-theory $\FK(A)\cong \FK(B)$
lifts to a $C[0,1]$-linear \Star{}isomorphism $A\cong B$.
\end{theorem}
The continuous fields classified by this theorem include fields that are nowhere locally trivial. 
It is for this reason that one needs to include infinitely many subintervals of $[0,1]$ in any complete invariant.
However it suffices to consider intervals whose endpoints belong to a countable dense subset of $[0,1]$.
The result does not extend to continuous fields of Kirchberg algebras if torsion is allowed, as we will explain shortly.

The main idea of our approach is to combine the following three crucial ingredients:
\begin{itemize}
\item
Eberhard Kirchberg's isomorphism theorem for non-simple nuclear $\Cuntz_\infty$-absorbing \Cstar{}al\-ge\-bras~\cite{Kirchberg},
\item
the results from \cite{DM}  which relate $\E$\nb-the\-ory over a second countable space $X$ with the corresponding version of $\KK$-theory and with $\E$\nb-the\-ory groups over finite approximating spaces of $X$,
\item
the universal coefficient theorem for accordion spaces from \cite{BK} (generalizing results from \cites{RS,Bonkat:Thesis,Restorff:Thesis,MN:Filtrated}) including a description of projective and injective objects in the target category of filtrated $\K$-theory.
\end{itemize}
The relevance of accordion spaces in this framework is due to the fact that sufficiently many non-Hausdorff finite approximating spaces of the unit interval are accordion spaces.

A major difficulty in any attempt to use the result of \cite{Kirchberg} is the computation of the group
$\KK(X;A,B)$ or at least a quotient of this group which allows to detect $\KK(X)$-equivalences.
In \cite{DM},  the second named author and Ralf Meyer proved a universal \emph{multi}-coefficient theorem (abbreviated UMCT) for separable $\Cont(X)$-algebras over a totally disconnected compact metrizable space~$X$. As a consequence, by Kirchberg's isomorphism theorem~\cite{Kirchberg}, separable stable continuous fields over such spaces whose fibres are UCT  Kirchberg algebras   are classified by an invariant the authors call \emph{filtrated $\K$-theory with coefficients}. This result is also implicit in \cite{Dadarlat-Pasnicu:Continuous_fields}.

The  filtrated $\K$-theory with coefficients of \cite{DM} comprises the \(\K\)\nobreakdash-the\-ory with coefficients (the $\Lambda$-modules defined in \cite{DL}, also called \emph{total $\K$-theory}) of all distinguished subquotients of the given field, along with the action of all natural maps between these groups. It is demonstrated in~\cite{DM}, generalising a result from~\cite{Dadarlat-Eilers:Bockstein}, that coefficients are necessary for such a classification result over any infinite metrizable compact space. This means that filtrated $\K$\nb-theory (without coefficients) can only be a classifying invariant on subclasses of fields with special $\K$-theoretical properties and this explains the need for additional assumptions in our results.
For comparison let us recall that the classification result of \cite{DE} is restricted to fields whose fibers have torsion-free $\K_0$-groups and vanishing $\K_1$-groups or vice versa.

The construction of an effective filtrated $\K$-theory with coefficients for \Cstar{}algebras over the unit interval remains
an open problem. In the final Section~\ref{sec:coefficients} we describe some of the technical difficulties that are encountered in potential constructions of such an invariant.

\section{Preliminaries}

In this section we summarize definitions and results by various authors which we shall use later. We make the convention $\N=\{1,2,3,\ldots\}$.

\subsection{\Cstar{}algebras over topological spaces}

Let~$X$ be any topological space. Recall from~\cite{MN:Bootstrap}:

\begin{definition}
A \emph{\Cstar{}algebra over~$X$} is a \Cstar{}algebra~$A$ equipped with a continuous map $\Prim(A)\to X$.
\end{definition}

\begin{definition}
Let~$A$ be a \Cstar{}algebra over~$X$. Let~$U\subseteq X$ be an open subset. Taking the preimage under the map $\Prim(A)\to X$, we may naturally associate the \emph{distinguished ideal} $A(U)\subseteq A$ to~$U$. A morphism of \Cstar{}algebras over~$X$ is a \Star{}homomorphism preserving all distinguished ideals.

A subset $Y\subset X$ is called \emph{locally closed} if it can be written as a difference $U\setminus V$ of two open subsets $V\subseteq U\subseteq X$. It can be shown that the \emph{distinguished subquotient} $A(Y)\defeq A(U)/A(V)$ is well-defined.
\end{definition}

We assume that~$X$ is locally compact Hausdorff in the following two definitions.

\begin{definition}
A \emph{$C_0(X)$}-algebra is a \Cstar{}algebra~$A$ equipped with a non-degenerate \Star{}homomorphism from $C_0(X)$ to the center of the multiplier algebra of~$A$. A morphism of $C_0(X)$-algebras is a $C_0(X)$-linear \Star{}ho\-mo\-mor\-phism.
\end{definition}

The category of \Cstar{}algebras over~$X$ and the category of $C_0(X)$-algebras are isomorphic (see \cite{MN:Bootstrap}*{Proposition~2.11}). We denote the category of separable \Cstar{}algebras over~$X$ by 
$\Cstarsep(X)$.

\begin{definition}
For $x\in X$ and a $C_0(X)$-algebra~$A$, we denote the quotient map $A\twoheadrightarrow A({x})$ onto the fiber by $\pi_x$. The algebra~$A$ is called \emph{continuous} if the function $x\mapsto\lVert\pi_x(a)\rVert$ is a continuous function on~$X$ for every $a\in A$.
\end{definition}

\subsection{Bivariant \texorpdfstring{$\K$}{K}-theory for \Cstar{}algebras over topological spaces}
Let~$X$ be a second countable topological space. 
Let us recall that \(\KKcat(X)\) is the triangulated category that extends $\KK$-theory to  separable \Cstar{}algebras over~$X$, see \cite{MN:Bootstrap}.
In~\cite{DM}, the second named author and Meyer define a version of $\E$\nb-the\-ory for separable \Cstar{}algebras over~$X$ and establish its basic properties. This construction yields a triangulated category $\Ecat(X)$ and a functor $\Cstarsep(X)\to\Ecat(X)$ which is characterized by a universal property. We recall two results which are of particular importance for us.

Let $\mathcal U=(U_n)_{n\in\N}$ be an ordered basis for the topology on~$X$. Denote by~$X_n$ the finite topological space, which arises as the $T_0$-quotient of~$X$ equipped with the topology generated by the set $\{U_1,\ldots,U_n\}$. Observe that we have a projective system of spaces $\cdots\twoheadrightarrow X_2\twoheadrightarrow X_1\twoheadrightarrow X_0$ together with compatible maps $X\twoheadrightarrow X_n$. By functoriality in the space variable, we obtain a projective sequence of triangulated categories $\bigl(\Ecat(X_n)\bigr)_{n\in\N}$ together with compatible functors $\Ecat(X)\to\Ecat(X_n)$.

\begin{proposition}[\cite{DM}*{Theorem 3.2}]
  \label{pro:finite_approximation}
Let~$A$ and~$B$ be separable \Cstar{}algebras over~$X$. Then there is a natural short exact sequence of $\Z/2$-graded Abelian groups
  \[
  \varprojlim\nolimits^1 \E_{*+1}(X_n;A,B)
  \into \E_*(X;A,B) \prto
  \varprojlim \E_*(X_n;A,B).
  \]
\end{proposition}

\begin{definition}
The bootstrap class~$\Bootstrap_\E$ consists of all separable \Cstar{}algebras that are equivalent in E-theory to a commutative
\Cstar{}algebra. The bootstrap class~$\Bootstrap_\E(X)$ consists of all separable \Cstar{}algebras over $X$ such that $A(U)$ belongs to~$\Bootstrap_\E$ for every open subset $U\subseteq X$. 
\end{definition}

\begin{proposition}[\cite{DM}*{Theorem 4.6}]
  \label{pro:invertibility_criterion}
Let~$A$ and~$B$ be separable \Cstar{}algebras over~$X$ belonging to the bootstrap class~$\Bootstrap_\E(X)$. An element in $\E(X;A,B)$ is invertible if and only if the induced map $\K_*\bigl(A(U)\bigr)\to\K_*\bigl(B(U)\bigr)$ is invertible for every open subset~$U$ of~$X$.
\end{proposition}

\subsection{Continuous fields of Kirchberg algebras}

In this subsection we assume that $X$ is a finite-dimensional, compact, metrizable topological space.

\begin{proposition}[]
  \label{pro:Kirchberg_class}
Let~$A$ be a separable continuous $C(X)$-algebra whose fibers are stable Kirchberg algebras. Then~$A$ is stable, nuclear and $\Cuntz_\infty$\nb-ab\-sorb\-ing. 
\end{proposition}

\begin{proof}
 Bauval shows in \cite{Bauval}*{Th\'{e}or\`{e}me ~7.2} that~$A$, being continuous and having nuclear fibers, is nuclear (in fact $C(X)$-nuclear). A combination of results by Blanchard, Kirchberg and R{\o}rdam in \cites{Blanchard-Kirchberg,Kirchberg-Rordam:Non-simple,Kirchberg-Rordam:Infinite,Rordam:Stable} implies that $A\otimes\Cuntz_\infty\otimes\Compacts\cong A$, see \cite{DE}*{Theorem~7.4} and \cite{HRW}.
\end{proof}

\begin{corollary}
  \label{cor:classification}
Let~$A$ and~$B$ be separable continuous $C(X)$-algebras whose fibers are stable Kirchberg algebras. Then every $\E(X)$-equivalence between~$A$ and~$B$ lifts to a $C(X)$-linear \Star{}isomorphism.
\end{corollary}

\begin{proof}
From \cite{DM}*{Theorem 5.4} we see that the given $\E(X)$-equivalence is induced by a $\KK(X)$-equivalence.
By Proposition~\ref{pro:Kirchberg_class}, we can apply Kirchberg's isomorphism theorem~\cite{Kirchberg}.
\end{proof}

\begin{proposition}
  \label{pro:bootstrap_class}
Let~$A$ be a separable nuclear continuous $C(X)$-algebra whose fibers satisfy the \textup{UCT}. Then $A$ belongs to the $\E(X)$-theoretic bootstrap class $\Bootstrap_\E(X)$.
\end{proposition}

\begin{proof}
This follows from \cite{Dadarlat:Fibrewise}*{Theorem~1.4} applied to every open subset of~$X$.
\end{proof}

\subsection{Filtrated \texorpdfstring{$\K$}{K}-theory over finite spaces}

In this subsection we assume that~$X$ is a finite $T_0$-space.

\begin{definition}
Let $\Ab^{\Z/2}$ be the category of $\Z/2$-graded Abelian groups and $\Z/2$-graded homomorphisms.
We denote the collection of non-empty, connected, locally closed subsets of~$X$ by $\Loclo(X)^*$. For $Y\in\Loclo(X)^*$, we have a functor $\FK^X_Y\colon\Ecat(X)\to\Ab^{\Z/2}$ taking~$A$ to $\K_*\bigl(A(Y)\bigr)$. Let $\Nattrafo^X$ be the $\Z/2$-graded pre-additive category whose object set is $\Loclo(X)^*$ and whose morphisms from~$Y$ to~$Z$ are the natural transformations from~$\FK^X_Y$ to~$\FK^X_Z$ regarded as functors from separable \Cstar{}algebras over~$X$ with $\Z/2$-graded morphism groups $\E_*(X;\blank,\blank)$ to $\Z/2$-graded Abelian groups with arbitrary group homomorphisms. The collection $\bigl(\FK_Y^X(A)\bigr)_{Y\in\Loclo(X)^*}$ has a natural graded module structure over $\Nattrafo^X$. This module is denoted by $\FK^X(A)$. Hence we have a functor $\FK^X\colon\Ecat(X)\to\GMod{\Nattrafo^X}$.
\end{definition}

\begin{remark}
If the space~$X$ is not too complicated, it is possible to describe the category $\Nattrafo^X$ in explicit terms. 
Suppose for instance that~$X$ is an accordion space in the sense of~\cite{BK}. Then $\Nattrafo^X$ is generated by six-term sequence transformations corresponding to inclusions of distinguished subquotients and an explicit generating list of relations can be given, see \cite{BK}.
\end{remark}

\begin{proposition}[\cite{BK}*{Theorem 8.9}]
  \label{pro:UCT_accordion}
Let~$X$ be an accordion space. Let~$A$ and~$B$ be separable \Cstar{}algebras over~$X$. Assume that~$A$ belongs to the bootstrap class~$\Bootstrap_\E(X)$. Then there is a natural short exact sequence of $\Z/2$-graded Abelian groups
\begin{multline*}
  \Ext^1_{\Nattrafo^X}\bigl(\FK^X(A),\FK^X(SB)\bigr)
  \into \E_*(X;A,B)\\ \prto
  \Hom_{\Nattrafo^X}\bigl(\FK^X(A),\FK^X(B)\bigr).
\end{multline*}
\end{proposition}

Here we have replaced $\KK_*(X;A,B)$ by $\E_*(X;A,B)$ in the original statement. This is possible as we explain in the following remark.

\begin{remark} It was shown in \cite{BK} that $\FK^X(A)$ has a projective
resolution
of length $1$ for every separable \Cstar{}algebra $A$ over $X$. Regarding $\FK^X$ as a
functor from $\Ecat(X)$, it is the universal $\mathcal{I}$-exact stable homological
functor,
where $\mathcal{I}$ is now the ideal in $\Ecat(X)$ consisting of all elements inducing
zero
maps in $\FK^X$. This is because $\KK(X;R,R)\cong \E(X;R,R)$ by \cite{DM}*{Theorem~5.5},
where $R\in \Bootstrap(X)$ is the representing object for $\FK^X$. 
Now the result follows from the general UCT of~\cite{MN:Filtrated}.
(Here $\Bootstrap(X)$ denotes the $\KK$-theoretic bootstrap class of \Cstar{}algebras over~$X$ defined by Meyer--Nest in~\cite{MN:Bootstrap} as the smallest class of \Cstar{}algebras containing all one-dimensional \Cstar{}algebras over~$X$ and closed under certain operations. If~$A$ is a nuclear \Cstar{}algebra over~$X$, then $A$ belongs to $\Bootstrap(X)$ if and only if it belongs to $\Bootstrap_\E(X)$.) 
\end{remark}

Not every $\Nattrafo^X$-module belongs to the range of the invariant~$\FK^X$. In particular, $\FK^X(A)$ is an \emph{exact} $\Nattrafo^X$-module for every \Cstar{}algebra $A$ over~$X$ as defined in \cite{MN:Filtrated}*{Definition~3.5}.

\begin{proposition}
\label{pro:projective_injective_modules}
Let $X$ be an accordion space and $M$ an $\Nattrafo^X$-module. Then~$M$ is projective\textup{/}injective if and only if~$M$ is exact and the $\Z/2$-graded Abelian group $M(Y)$ is projective\textup{/}injective for every $Y\in\Loclo(X)^*$.
\end{proposition}

\begin{proof}
The statement about projective modules is proven in~\cite{BK}. The claim about injective modules follows from a dual argument.
\end{proof}

\section{Finite approximations of the unit interval}

Let~$I=[0,1]$ be the unit interval. Choose, once and for all, a dense sequence $(d_n)_{n\in\N}$ in~$I$. For convenience, we may assume $d_m\neq d_n$ for $m\neq n$ and $d_n\not\in\{0,1\}$ for all $n\in\N$. Consider the ordered subbasis $\mathcal V=(V_n)_{n\in\N}$ for the topology on~$I$ given by $V_{2n-1}=[0,d_n)$ and $V_{2n}=(d_n,1]$; denote by~$I_n$ the $T_0$-quotient of~$I$ equipped with the topology generated by the set $\{V_1,\ldots,V_{2n}\}$.

Let~$A$ and~$B$ be separable \Cstar{}algebras over~$I$. Since the spaces~$I_n$ form a cofinal family in the projective sequence of approximations corresponding to the basis generated by the subbasis~$\mathcal V$ above, Proposition~\ref{pro:finite_approximation} yields a short exact sequence
  \begin{equation}
  \label{eq:lim_one_sequence0}
  \varprojlim\nolimits^1 \E_{*+1}(I_n;A,B)
  \into \E_*(I;A,B) \prto
  \varprojlim \E_*(I_n;A,B).
  \end{equation}
We are therefore interested in the computation of the groups $\E_*(I_n;A,B)$.

\begin{lemma}
The spaces~$I_n$ are accordion spaces.
\end{lemma}

\begin{proof}
For a given natural number $n\in\N$, we order the set $\{d_1,\ldots,d_n\}$ by writing $\{d_1,\ldots,d_n\}=\{e_1,\ldots,e_n\}$ where $e_k<e_{k+1}$ for $1\leq k<n$. Then we have
\[
I_n=\bigl\{[0,e_1),\{e_1\},(e_1,e_2),\{e_2\},(e_2,e_3),\ldots,\{e_n\},(e_n,1]\bigr\}.
\]
Denoting $u_0=[0,e_1)$, $u_k=(e_k,e_{k+1})$ for $1\leq k<n$, $u_n=(e_n,1]$ and $c_k=\{e_k\}$ for $1\leq k\leq n$, a basis for the topology on~$I_n$ given by the family of open subsets
\[
 \bigl\{\{u_k\}\mid 0\leq k\leq n\bigr\}\cup\bigl\{\{u_k,c_k,u_{k+1}\}\mid 0\leq k<n\bigr\}.
\]
Hence~$I_n$ is an accordion space of a specific form, the Hasse diagram of the specialization order of which is indicated in the diagram below.
\[
\bullet\rightarrow\bullet\leftarrow\bullet\rightarrow\bullet\leftarrow\bullet\rightarrow\cdots\leftarrow\bullet\rightarrow\bullet\leftarrow\bullet\rightarrow\bullet\leftarrow\bullet\qedhere
\]
\end{proof}

For $n\in\N$, we briefly write $\Nattrafo_n$ for $\Nattrafo^{I_n}$ and $\FK_n(A)$ for $\FK^{I_n}(A)$.

Assume that~$A$ belongs to the bootstrap class~$\Bootstrap_\E(I)$. By Proposition~\ref{pro:UCT_accordion}, for every $n\in\N$, we have a short exact sequence
\begin{multline}
  \label{eq:UCT}
  \Ext^1_{\Nattrafo_n}\bigl(\FK_n(A),\FK_n(SB)\bigr)
  \into \E_*(I_n;A,B)\\ \prto
  \Hom_{\Nattrafo_n}\bigl(\FK_n(A),\FK_n(B)\bigr).
\end{multline}

\begin{definition}\label{def:FK}
Let~$A$ be a \Cstar{}algebra over $[0,1]$. The \emph{filtrated $\K$-theory of~$A$} consists of the $\Z/2$-graded Abelian groups $\K_*\bigl(A(Y)\bigl)$ for all locally closed subintervals $Y\subseteq I$ together with the graded group homomorphisms in the six-term exact sequence $\K_*\bigl(A(U)\bigl)\to\K_*\bigl(A(Y)\bigl)\to\K_*\bigl(A(Y\setminus U)\bigl)\to\K_{*+1}\bigl(A(U)\bigl)$ for every relatively open subinterval~$U$ of a locally closed interval~$Y\subseteq I$ with the property that the set $Y\setminus U$ is connected.
A \emph{homomorphism} from $\FK(A)$ to $\FK(B)$ is a family of graded group homomorphisms $$\{\varphi_Y\colon\K_*\bigl(A(Y)\bigl)\to\K_*\bigl(B(Y)\bigl)\}_Y$$ such that for all pairs  $U\subset Y$ as above,  all squares in the diagram
\[
\xymatrix{
\K_*\bigl(A(U)\bigl)\ar[r]\ar[d]^{\varphi_U} & \K_*\bigl(A(Y)\bigl)\ar[r]\ar[d]^{\varphi_Y} & \K_*\bigl(A(Y\setminus U)\bigl)\ar[r]\ar[d]^{\varphi_{Y\setminus U}} & \K_{*+1}\bigl(A(U)\bigl)\ar[d]^{\varphi_U} \\
\K_*\bigl(B(U)\bigl)\ar[r] & \K_*\bigl(B(Y)\bigl)\ar[r] & \K_*\bigl(B(Y\setminus U)\bigl)\ar[r] & \K_{*+1}\bigl(B(U)\bigl)
}
\]
commute.

The $\Z/2$-graded Abelian group of homomorphisms from $\FK(A)$ to $\FK(B)$ is denoted by $\Hom_\Nattrafo\bigl(\FK(A),\FK(B)\bigr)$.
\end{definition}
We note that one may consider a variation $\FK'(A)$ of $\FK(A)$ where only intervals with  endpoints from the sequence $(d_n)_{n\in\N}$ and $0,1$
are used. It is not hard to show that the restriction map
$\Hom_\Nattrafo\bigl(\FK(A),\FK(B)\bigr)\to \Hom_\Nattrafo\bigl(\FK'(A),\FK'(B)\bigr)$ is bijective.
It follows that
\[
\Hom_\Nattrafo\bigl(\FK(A),\FK(B)\bigr)=\varprojlim\Hom_{\Nattrafo_n}\bigl(\FK_n(A),\FK_n(B)\bigr).
\]
\begin{remark}
We can regard $\FK(A)$ as a $\Z/2$-graded module over a $\Z/2$-graded pre-additive category $\Nattrafo$ with objects the locally closed subintervals of~$I$ and morphisms generated by elements $i_U^Y$, $r_Y^{Y\setminus U}$, $\delta_{Y\setminus U}^U$ for every relatively open subinterval~$U$ of a locally closed interval~$Y\subseteq I$ such that $Y\setminus U$ is connected. Regardless of the relations among these generators, homomorphisms from $\FK(A)$ to $\FK(B)$ would then simply be graded module homomorphisms. This justifies the notation $\Hom_\Nattrafo\bigl(\FK(A),\FK(B)\bigr)$.
\end{remark}

\section{Classification results}

We are now ready to put together the facts from the previous sections to derive classification results.

Applying inverse limits to the UCT-sequences \eqref{eq:UCT}, and using that  $\varprojlim\nolimits^1$ is a derived functor of $\varprojlim$, we obtain the exact sequence
\begin{multline}
  \label{eq:lim_one_sequence}
0
\to\varprojlim\Ext^1_{\Nattrafo_n}\bigl(\FK_n(A),\FK_n(SB)\bigr)
\to\varprojlim\E_*(I_n;A,B)\\
\to\Hom_\Nattrafo\bigl(\FK(A),\FK(B)\bigr)
\xrightarrow{d}\varprojlim\nolimits^1\Ext^1_{\Nattrafo_n}\bigl(\FK_n(A),\FK_n(SB)\bigr).
\end{multline}

\begin{definition}
Let $\mathcal{K}ir(I)$ denote the class of separable continuous $C(I)$\nb-al\-ge\-bras whose fibers are stable Kirchberg algebras satisfying the \textup{UCT}.
\end{definition}

\begin{theorem}
  \label{thm:general_classification}
Let $\mathscr C$ be a subclass of~$\mathcal{K}ir(I)$ such that for all~$A$ and~$B$ in~$\mathscr C$, the map~$d$ in \eqref{eq:lim_one_sequence} vanishes. Then, for all~$A$ and~$B$ in~$\mathscr C$, the map $\E_*(I;A,B)\to \Hom_\Nattrafo\bigl(\FK(A),\FK(B)\bigr)$ is surjective and every isomorphism $\FK(A)\cong\FK(B)$ lifts to a $C(I)$-linear \Star{}iso\-mor\-phism.
\end{theorem}
\begin{proof}
Let $\alpha\in \Hom_\Nattrafo\bigl(\FK(A),\FK(B)\bigr)$.
If $d(\alpha)=0$, we can use the exact sequences \eqref{eq:lim_one_sequence} and \eqref{eq:lim_one_sequence0}
to lift $\alpha$ to an element $\tilde{\alpha}\in \E(I;A,B)$.
If~$\alpha$ was an isomorphism, then $\tilde{\alpha}$ is an $\E(I)$-equivalence by
Proposition~\ref{pro:invertibility_criterion}.
We conclude the proof by applying Corollary~\ref{cor:classification}.
\end{proof}
\begin{remark}
Theorem~\ref{thm:general_classification} does not hold for the whole class $\mathscr C=\mathcal{K}ir(I)$
as shown by Example 6.5 from \cite{DM}. If the fibers of $A$ and $B$ have torsion in $\K$-theory, then the map
$\E_*(I;A,B)\to \ker(d) \subset \Hom_\Nattrafo\bigl(\FK(A),\FK(B)\bigr)$ is typically not surjective.
\end{remark}

We will now verify the hypotheses of Theorem~\ref{thm:general_classification} for certain classes of \Cstar{}algebras over $[0,1]$. Our first example yields (in particular) a proof of Theorem~\ref{thm:main-result}.
\begin{example} (\emph{Proof of Theorem}~\ref{thm:main-result}.)
By Proposition~\ref{pro:projective_injective_modules}, the conclusion of Theorem~\ref{thm:general_classification} holds for the class~$\mathscr C$ of $C(I)$-algebras~$A$ in~$\mathcal{K}ir(I)$ for which $\K_*\bigl(A(Y)\bigr)$ is a divisible Abelian group for every locally closed interval $Y\subseteq I$. By the K\"{u}nneth formula for tensor products, the class~$\mathscr C$ contains all objects in~$\mathcal{K}ir(I)$ which are stable under tensoring with the universal UHF-algebra $M_\mathbb{Q}$. 
Let $A$ be as in Theorem~\ref{thm:main-result}. Since $\K_*(A(x))\cong\K_*(A(x))\otimes \mathbb{Q}$ it follows
that $A(x)\cong A(x)\otimes M_\mathbb{Q}$, for all $x\in I$, by the Kirchberg--Phillips classification theorem.
We conclude the argument by noting that if each fiber of a $C(I)$\nb-algebra $A$ is stable under tensoring with the universal UHF-algebra~$M_\mathbb{Q}$, then so is~$A$ itself by~\cite{HRW}.
\end{example}

\begin{example}
Again by Proposition~\ref{pro:projective_injective_modules}, the conclusion of Theorem~\ref{thm:general_classification} holds for the class~$\mathscr C$ of $C(I)$-algebras~$A$ in~$\mathcal{K}ir(I)$ for which $\K_*\bigl(A(Y)\bigr)$ is a free Abelian group for every locally closed interval $Y\subseteq I$ because the $\Ext^1$-terms in \eqref{eq:lim_one_sequence} vanish.
\end{example}

\begin{example}
Fix $i\in\{0,1\}$. Consider the class~$\mathscr C$ of $C(I)$-algebras~$A$ in~$\mathcal{K}ir(I)$ which satisfy $\K_i\bigl(A(Y)\bigr)=0$ for every locally closed interval $Y\subseteq I$. For parity reasons, the $\Ext^1$-terms in \eqref{eq:lim_one_sequence} vanish. Hence the class~$\mathscr C$ satisfies the condition of Theorem~\ref{thm:general_classification}.
\end{example}

\begin{remark}
Fix $i\in\{0,1\}$. It follows from the main result in~\cite{DE} that the condition of Theorem~\ref{thm:general_classification} is also satisfied for the class~$\mathscr C$ of $C(I)$-algebras~$A$ in~$\mathcal{K}ir(I)$ whose fibers have vanishing $\K_d$-groups and torsion-free $\K_{d+1}$-groups. However, we have not been able to reprove this by an independent, purely $\K$-theoretical argument. 
\end{remark}

\section{A remark on coefficients}
  \label{sec:coefficients}
In order to get a classification result without any $\K$-theoretical assumptions, one expects, as indicated in the introduction, to need some version of filtrated $\K$-theory with coefficients for \Cstar{}algebras over the unit interval. This requires, to begin with, the correct definition of filtrated $\K$-theory with coefficients for \Cstar{}algebras over accordion spaces. It was observed in \cite{ERR} that, already over the two-point Sierpi\'{n}ski space~$S$, the na\"{i}ve candidate for such a definition---using the corresponding six-term sequence of $\Lambda$-modules---produces an invariant which lacks desired properties such as a UMCT.

We argue that, in order to give a fully satisfactory definition of filtrated $\K$-theory with coefficients for \Cstar{}algebras over $S$, one has to allow all finitely generated, indecomposable exact six-term sequences of Abelian groups as coefficients---just as  all finitely generated, indecomposable Abelian groups as coefficients are needed in the UMCT of \cite{DL}. It is easy to see that there is a countable number of isomorphism classes of such six-term sequences. However, unlike in the case of Abelian groups, it follows from the main result in \cite{Ringel_Schmidmeier} that their classification is controlled $\Z/p$-wild for every prime~$p$. This wildness phenomenon seems to make filtrated $\K$-theory with (generalized) coefficients as sketched above very hard to compute explicitly, limiting its r\^{o}le in the theory to a rather theoretical one.

We conclude by remarking that recent results of Eilers, Restorff and Ruiz in~\cite{Eilers-Restorff-Ruiz:Autos} indicate that additional $\K$-theoretical assumptions allow the usage of a smaller, more concrete invariant.

\subsection*{Acknowledgement}
The first named author wishes to express his gratitude towards the Department of Mathematics at Purdue University and, in particular, its operator algebra group for the kind hospitality offered during a visit in spring 2012, where the present work was initiated.

\end{document}